\documentclass[11pt]{article} 

\usepackage[utf8]{inputenc} 

\usepackage{geometry} 
\usepackage{graphicx} 

\usepackage{array} 
\usepackage{paralist} 
\usepackage{verbatim} 
\usepackage{subfig} 
\usepackage{amsthm}
\usepackage{amsmath}
\usepackage{amsfonts}
\usepackage{amssymb}
\usepackage{cite}
\usepackage{cleveref}

\newtheorem*{theorem*}{Theorem}
\newtheorem{theorem}{Theorem}[section]
\newtheorem{lemma}[theorem]{Lemma}
\newtheorem{corollary}[theorem]{Corollary}
\newtheorem{proposition}[theorem]{Proposition}
\theoremstyle{remark}
\newtheorem{remark}[theorem]{Remark}

\title{On certain tilting modules for $SL_2$ II}
\author{Samuel Martin}
  
\begin{document}
\maketitle
       
\begin{abstract}
	Following \cite{Sam}, we determine exactly the highest weights for which a tensor product of two induced modules is a tilting module, for the algebraic group $SL_2$ over an algebraically closed field of positive characteristic.
\end{abstract}

\section{Introduction}

Let $G$ be the group $SL_2(k)$, where $k$ is an algebraically closed field of positive characteristic, and denote by $\nabla(r)$ the induced module of highest weight $r$. The tensor product $\nabla(r)\otimes\nabla(s)$, has been studied previously by Cavallin in \cite{Cavallin}, where he used the identification of $\nabla(r)$ as the $r^\text{th}$ symmetric power of the two dimensional natural module $E$ to give a decomposition of the polynomial $GL_2(k)$ modules $S^rE\otimes S^sE$. In \cite{Sam}, a description of the weights $r,s$ such that $\nabla(r)\otimes\Delta(s)$ is a tilting module is given, and in \cite{Sam+Steve} the decomposition into indecomposable summands of these modules is given. It turns out that this decomposition depends entirely upon the tilting decomposition of those $\nabla(r)\otimes\Delta(s)$ that are tilting modules. Here, we will use the results given in \cite{Sam} to determine the values of $r$ and $s$ for which $\nabla(r)\otimes\nabla(s)$ is a tilting module.
\\
\begin{remark}\label{isoremark}
We first remark that from \cite[Theorem~1.1]{Sam} we have that for $a \in \{1,\dots, p-1\}$ and $n \in \mathbb{N}$, the module $\nabla(ap^n -1 )$ is a tilting module, and thus we have an isomorphism

$$\nabla(ap^n-1) \cong \Delta(ap^n-1).$$

It follows that $\nabla(r)\otimes\nabla(ap^n - 1)$ is isomorphic to $\nabla(r)\otimes\Delta(ap^n -1)$ for all $r \in \mathbb{N}$, and so $\nabla(r)\otimes\nabla(ap^n - 1)$ is a tilting module if and only if $\nabla(r)\otimes\Delta(ap^n -1)$ is a tilting module.
\end{remark}
\begin{remark}\label{prestricted}
Let $r,s$ be $p$-restricted weights, that is $r,s \in \{0, \ldots, p-1\}$. Then it is well known that we have isomorphisms $\nabla(r)\cong\Delta(r)\cong L(r)\cong T(r)$, and so in this case $\nabla(r)\otimes\nabla(s)$ is a tilting module.
\end{remark}

We now state the main theorem of this paper.
\begin{theorem}\label{mainTheorem}
Let $r,s \in \mathbb{N}$ be such that they are not both $p$-restricted. The module $\nabla(r)\otimes\nabla(s)$ is a tilting module if and only if at least one of $r$ and $s$ is equal to $ap^n -1$, for some $a \in \{1, \ldots, p-1\}$ and $n \in \mathbb{N}$, and the other is strictly less than $p^{n+1} - 1$.
\end{theorem}

Following \Cref{isoremark}, the converse statement is easy to prove: If $s = ap^n-1$ and $r <p^{n+1} - 1$, then $\nabla(r)\otimes\nabla(s)\cong\nabla(r)\otimes\Delta(s)$, and by \cite[Theorem~1.1]{Sam} this module is a tilting module. Similarly, if $r=ap^n-1$ and $s < p^{n+1} - 1$ we have  $\nabla(r)\otimes\nabla(s)\cong\Delta(r)\otimes\nabla(s)$, and again, this is a tilting module. The rest of this paper will be devoted to proving that these are the only modules for which $\nabla(r)\otimes\nabla(s)$ is a tilting module.

\begin{remark}
We note that since the dual of a tilting module is a tilting module, \Cref{mainTheorem} also gives us the result for the tensor product of two Weyl modules.
\end{remark}

I am grateful to Stephen Donkin for the following observation.
\begin{remark}
The question of whether $\nabla(r)\otimes\nabla(s)$ is a tilting module is equivalent to asking when the module $S^rE \otimes S^sE$ is a tilting module in the category of polynomial $GL_2(k)$ modules. This is equivalent to asking when the module is both injective and projective in this category, and the problem of determining which indecomposable modules are both projective and injective in the category of polynomial $GL_n$ modules was considered in \cite{DonkinDeVisscher}. In the cases $n=2,3$ a complete solution is given, so one could also use this with the decomposition given in \cite{Cavallin} to determine whether $\nabla(r)\otimes\nabla(s)$ is a tilting module.
\end{remark}

\section{Background}
Before proving \Cref{mainTheorem}, we fix some terminology and give an overview of the results we will need. Throughout, $k$ will be an algebraically closed field of characteristic $p > 0$, and $G$ will be the affine algebraic group $SL_2(k)$. Let $B$ be the Borel subgroup of $G$ consisting of lower triangular matrices and containing the maximal torus $T$ of diagonal matrices. Let $X(T)$ be the weight lattice, which we associate with $\mathbb{Z}$ in the usual manner, so that if $x_t = \text{diag}(t, t^{-1}) \in T$ then $r(x_t) = t^r$. Under this association the set of dominant weights $X^{+}$ corresponds to the set $\mathbb{N}\cup\{0\}$.
\\
\\
By a \lq module\rq, we will always mean a finite dimensional, rational $G$-module. Let $F:G\longrightarrow G$ denote the usual Frobenius morphism, and denote by $G_1$ its kernel. For any module $V$, we denote by $V^F$ the Frobenius twist of $V$.
\\
\\
For a rational module $V$ and weight $r$, we define the $r$ weight space of $V$ to be
$$ V_r := \{ v \in V\,:\,xv = r(x)v \ \text{for all}\ x \in T\, \}.$$
We say that $r$ is a weight of $V$ if $V_r$ is non-zero. As a $T$ module, $V$ has a decomposition as a direct sum of it's non-zero weight spaces. Denote by $E$ the natural two dimensional $G$ module, with basis elements $x_1$ and $x_2$. For any $r \in \mathbb{N}$ we denote by $S^r E$ the $r^\text{th}$ symmetric power of $E$, which has basis consisting of monomials in $x_1$ and $x_2$ of degree $r$.
\\
\\
Let $k_r$ be the one dimensional $B$ module on which $T$ acts via $r \in \mathbb{Z}$, and let $\nabla(r)$ be the induced module $\text{Ind}_B^G(k_r)$. Then $\nabla(r)$ is finite dimensional and is zero when $r$ is not dominant. When $r$ is dominant we have the isomorphism $\nabla(r) \cong S^rE$. Let $\Delta(s)$ be the Weyl module of highest weight $s$, for which we have $\Delta(s) = \nabla(s)^{*}$. If $m_+ \in \Delta(s)$ is a highest weight vector, then $\Delta(s)$ has a basis given by the elements $m_+, f_1m_+, \ldots, f_sm_+$, where $f_i$ is the divided power operator
$$f_i := 1\otimes_\mathbb{Z} \frac{f^i}{i!}$$
in $U_k = k\otimes_{\mathbb{Z}}U_{\mathbb{Z}}(\mathfrak{g})$, the algebra of distributions of $G$. Here, $U_\mathbb{Z}(\mathfrak{g})$ is the Kostant $\mathbb{Z}$-form of the universal enveloping algebra of the Lie algebra $\mathfrak{g}$ of $G$, and $f$ is the usual basis element of $\mathfrak{g}$. The action of $f_i$ on a tensor product of rational $G$ modules is given by 

$$ f_i(x\otimes y) = \sum_{a+b = i} f_ax \otimes f_b y.$$

By a tilting module we mean a module which has both a good filtration and a Weyl filtration as defined in \cite{Donkin}. For each dominant weight $r$ there exists a unique indecomposable tilting module of highest weight $r$, which we denote by $T(r)$, and the dimension of its $r$ weight space is $1$. The modules $T(r)$ form a complete set of inequivalent indecomposable tilting modules \cite[Theorem~(1.1)]{Donkin}, and the tensor product of two tilting modules is also a tilting module \cite[Proposition~1.2(i)]{Donkin}.
\\
\\

We recall the following key results from \cite{Sam}.

\begin{lemma}[\cite{Sam} Proposition~2.1] \label{tensorE}

There exists a short exact sequence

$$ 0 \longrightarrow \nabla(r-1) \longrightarrow \nabla(r)\otimes E \longrightarrow \nabla(r+1) \longrightarrow 0,$$

and this is split if and only if $p$ does not divide $r+1$. \qed
\end{lemma}

For the next lemma, let $V$ be a rational $SL_2(k)$ module, and denote by $H^0(G_1, V)$ the submodule of $G_1$ fixed points of $V$. We have the identity $H^0(G_1, V_1\otimes V_2^F) \cong H^0(G_1, V_1)\otimes V_2^F$.

\begin{lemma}[\cite{Sam} Lemma~2.4]\label{tiltingTwist2}
Let $V$ be a tilting module, and define the module $W$ by $H^0(G_1, V) = W^F$. Then $W$ is a tilting module. \qed
\end{lemma}

We will use this result in the following way. We have that $H^0(G_1, T(2p-2)) = L(0)$ and for $1 \leq t < 2p-2$ we have $H^0(G_1, T(t)) = 0$. Now $\nabla(p-1)\otimes\nabla(p-1)$ is a tilting module of highest weight $2p-2$. Its decomposition into indecomposable tilting modules contains $T(2p-2)$ exactly once, and since $\nabla(0)$ appears in the good filtration of $T(2p-2)$, there is no summand isomorphic to $T(0)$ in the decomposition of $\nabla(p-1)\otimes\nabla(p-1)$. Then we have that $H^0(G_1, \nabla(p-1)\otimes\nabla(p-1)) = L(0)$, and thus for any rational module $W$ we have

$$H^0(G_1, \nabla(p-1)\otimes\nabla(p-1)\otimes W^F) \cong W^F.$$

See the proof of this lemma in \cite{Sam} for further details.
\section{Main Result}
In this section we will show that the only $r,s$ for which $\nabla(r)\otimes\nabla(s)$ is a tilting module are those given in \Cref{mainTheorem}. We will break this up into two cases.
\\
\\
For the first case, we consider $r$ and $s$ with one equal to $ap^n-1$ for $a \in \{1,\ldots, p-1\}$ and $n \in \mathbb{N}$, and the other greater than or equal to $p^{n+1}$. Let's say $r = ap^n - 1$. As in \Cref{isoremark} we have an isomorphism $\nabla(r)\otimes\nabla(s)\cong\Delta(r)\otimes\nabla(s)$, and by \cite[Theorem~1.1]{Sam} this module is not a tilting module.
\\
\\
The remaining cases can be given by $r \in \{ ap^n, ap^n + 1, \ldots, (a+1)p^n -2 \}$ and $s \in \{bp^n, bp^n +1, \ldots (b+1)p^n -2 \}$, where $a,b \in \{0,1, \ldots, p-1\}$, not both equal to $0$, and $n\in\mathbb{N}$. The key result we will need is the following.

\begin{lemma}\label{nablap}
For $a, b \in \mathbb{N}\cup\{0\}$, not both equal to $0$, the module $\nabla(ap)\otimes\nabla(bp)$ is not a tilting module.
\end{lemma}
\begin{proof}
The case when either $a=0$ or $b=0$ can be seen as a special case of \cite[Theorem~1.1]{Sam}. For the remaining cases we first note that the indecomposable tilting module $T((a+b)p)$ has a Weyl filtration with exactly one section isomorphic to $\Delta((a+b)p)$, and all others isomorphic to $\Delta((a+b)p - 2i)$ for some $i \in \mathbb{N}$. Since $\text{Ext}^1_G(\Delta((a+b)p), \Delta((a+b)p-2i)) = 0$ \cite[(3.1)(1)]{Erdmann}, we have in particular that $T((a+b)p)$ contains a submodule isomorphic to $\Delta((a+b)p)$.
\\
\\
Now, we will assume $a \geq b$, so that $\text{Ch}\nabla(ap)\otimes\nabla(bp) = \chi((a+b)p) + \chi ((a+b)p-2) + \cdots + \chi((a-b)p)$, and the weight space $(\nabla(ap)\otimes\nabla(bp))_{(a+b)p}$ is one dimensional, and given by the span of the vector $x_1^{ap}\otimes x_1^{bp}$. For this vector we have 

$$f(x_1^{ap}\otimes x_1^{bp}) = fx_1^{ap} \otimes x_1^{bp} + x_1^{ap}\otimes fx_1^{bp} = 0, $$

using that $fx_1^{ap} =apx_1^{ap-1}x_2 = 0$ and similarly $fx_1^{bp} = 0$. Now if $\nabla(ap)\otimes\nabla(bp)$ were a tilting module, it would have exactly one summand isomorphic to $T((a+b)p)$, and so also one submodule isomorphic to $\Delta((a+b)p)$. Such a submodule would be generated by a vector of weight $(a+b)p$, but we have just shown that the only vector of weight $(a+b)p$ does not generate such a submodule, and we conclude that $\nabla(ap)\otimes\nabla(bp)$ is not a tilting module.

\end{proof}

We obtain the following corollary.

\begin{corollary}\label{base}
Let $a, b$ be as above, then for $r \in   \{ap, ap+1, \ldots, ap + p-2\}$  and $s \in \{bp, bp+1, \ldots, bp + p-2\}$, the module $\nabla(r)\otimes\nabla(s)$ is not tilting.
\end{corollary}
\begin{proof}
For a contradiction, suppose that for some $r$ and $s$ we have that $\nabla(r)\otimes\nabla(s)$ is tilting, and choose $r$ and $s$ so that $r+s$ is minimal. Then we have that $\nabla(r)\otimes E\otimes\nabla(s)$ is tilting, but since $p$ does not divide $r + 1$, by \Cref{tensorE} we have that this module is isomorphic to 

$$ \nabla(r-1)\otimes\nabla(s) \oplus \nabla(r+1)\otimes\nabla(s),$$

and so each summand is a tilting module. Similarly we have that $\nabla(r)\otimes\nabla(s-1)$ is a tilting module, so by minimality we have that $\nabla(ap)\otimes\nabla(bp)$ is tilting, contradicting \Cref{nablap}.
\end{proof}

\begin{remark}\label{twistRemark}
Before giving the final result, we note that if $r = p-1 + pr'$ then we have the isomorphism $\nabla(r)\cong\nabla(p-1)\otimes\nabla(r')^F$ \cite[Proposition~II.3.19]{Jantzen}.
\end{remark}

\begin{proposition}
For $r \in \{ ap^n, ap^n + 1, \ldots, (a+1)p^n -2 \}$ and $b \in \{bp^n, bp^n +1, \ldots (b+1)p^n -2 \}$, where $a,b \in \{1, \ldots, p-1\}$ and $n \in \mathbb{N}$, the module $\nabla(r)\otimes\nabla(s)$ is not a tilting module.
\end{proposition}
\begin{proof}
We prove this by induction on $n$, where the base case $n=1$ is given by \Cref{base}. Write $r$ and $s$ uniquely in the form $r = ap^n + a_0p + r_0$ and $s = bp^n + b_0 + s_0$, with $r_0, s_0 \in \{0, \ldots, p-1\}$ and $0 \leq a_0, b_0 < p^{n-1} - 1$. Now if neither $r_0$ nor $s_0$ is equal to $p-1$, then the result holds by \Cref{base}. Without loss of generality, we assume $r_0 = p-1$.
\\
\\
We consider the cases $s_0 = p-1$ and $s_0 \neq p-1$ separately. If $s_0 = p-1$ then we have

$$ \nabla(r)\otimes\nabla(s) \cong \nabla(p-1)\otimes\nabla(p-1)\otimes\Big(\nabla(ap^{n-1} + a_0)\otimes\nabla(bp^{n-1} + b_0)\Big)^F,$$

from \Cref{twistRemark}. By induction we have that $\nabla(ap^{n-1} + a_0)\otimes\nabla(bp^{n-1} + b_0)$ is not a tilting module, so by \Cref{tiltingTwist2}, neither is  $\nabla(r)\otimes\nabla(s)$.
\\
\\
For $s_0 < p-1$, suppose that $\nabla(r)\otimes\nabla(s)$ is a tilting module, and choose $s_0$ maximal. Then since $p$ does not divide $s_0 + 1$ we have that the tilting module $\nabla(r)\otimes E\otimes\nabla(s)$ is isomorphic to the direct sum 

$$\nabla(r)\otimes\nabla(s-1) \oplus \nabla(r)\otimes\nabla(s+1)$$

and each summand is a tilting module. Thus we take $s_0 = p-2$, but then by the above $\nabla(r)\otimes\nabla(s+1)$ is a tilting module, contradicting the case when $s_0 = p-1$. We conclude that for all $0 \leq s_0 \leq p-1$, the module $\nabla(r)\otimes\nabla(s)$ is not a tilting module.

\end{proof}

We have shown that for all $r$ and $s$ which are not as described in \Cref{mainTheorem}, the module $\nabla(r)\otimes\nabla(s)$ is not a tilting module, thereby completing the proof.

\bibliography{tiltingbib}
\bibliographystyle{plain}

\end{document}